\documentclass[11pt, a4paper]{amsart}

\usepackage{amsmath}
\usepackage{amsfonts}
\usepackage{amssymb}
\usepackage{amsthm}
\usepackage{mathrsfs}
\usepackage{graphics}
\usepackage[all]{xy}
\usepackage{mathtools}
\newcommand{\seq}{\coloneqq}
\usepackage[
bookmarks=false,
colorlinks=true,
debug=true,
naturalnames=true,
pdfnewwindow=true,
citecolor=blue,
linkcolor=blue,
urlcolor = blue]{hyperref}
\usepackage{comment}
\mathtoolsset{showonlyrefs=true}
\usepackage{yfonts}
\usepackage{ulem}

\newenvironment{red}{\relax\color{red}}{\relax}
\newenvironment{blue}{\relax\color{blue}}{\relax}

\newcommand{\ber}{\begin{red}}
\newcommand{\er}{\end{red}}
\newcommand{\beb}{\begin{blue}}
\newcommand{\eb}{\end{blue}}

\newtheorem{Thm}{Theorem}[section]
\newtheorem{Cor}[Thm]{Corollary}
\newtheorem{Prop}[Thm]{Proposition}
\newtheorem{Lem}[Thm]{Lemma}

\newtheorem{Conj}[Thm]{Conjecture}
\theoremstyle{definition}
\newtheorem{Def}[Thm]{Definition}
\newtheorem{Rem}[Thm]{Remark}
\newtheorem{Ex}[Thm]{Example}

\numberwithin{equation}{section}

\newcommand{\arxiv}[1]{\href{http://arxiv.org/abs/#1}{\texttt{arXiv:#1}}}

\newcommand{\ol}{\overline}

\newcommand{\ep}{\epsilon}

\newcommand{\fsl}{\mathfrak{sl}}
\newcommand{\SG}{\mathfrak{S}}

\newcommand{\Z}{\mathbb{Z}}
\newcommand{\Q}{\mathbb{Q}}
\newcommand{\C}{\mathbb{C}}

\newcommand{\Cc}{\mathscr{C}}

\newcommand{\Nn}{\mathscr{N}}
\newcommand{\Aa}{\mathscr{A}}
\newcommand{\Tt}{\mathscr{T}}

\newcommand{\cM}{\mathcal{M}}
\newcommand{\cY}{\mathcal{Y}}

\newcommand{\tc}{\tilde{c}}
\newcommand{\tC}{\tilde{C}}
\newcommand{\tn}{{\tilde{n}}}
\newcommand{\te}{{\tilde{e}}}
\newcommand{\txi}{{\tilde{\xi}}}
\newcommand{\tE}{\tilde{E}}

\newcommand{\wh}{\widehat}
\newcommand{\hI}{\wh{I}}

\newcommand{\hA}{\wh{\mathcal{A}}}

\newcommand{\bc}{\mathbf{c}}

\newcommand{\id}{\mathsf{id}}

\newcommand{\wt}{\mathrm{wt}}

\newcommand{\ev}{\mathop{\mathrm{ev}}\nolimits}

\newcommand{\Lsl}{L\mathfrak{sl}}

\makeatletter
\@namedef{subjclassname@2020}{%
  \textup{2020} Mathematics Subject Classification}
\makeatother

\title{Inflations among quantum Grothendieck rings of type A}
\date{\today}

\author{Ryo Fujita}
\address{Research Institute for Mathematical Sciences, Kyoto University, Kitashirakawa-Oiwake-cho, Sakyo, Kyoto, 606-8502, Japan}
\email{rfujita@kurims.kyoto-u.ac.jp}

\dedicatory{Dedicated to Vyjayanthi Chari on her 65th birthday}

\begin{document}

\maketitle

\begin{abstract}
We introduce a collection of injective homomorphisms among the quantum Grothendieck rings of finite-dimensional modules over the quantum loop algebras of type $\mathrm{A}$.
In the classical limit, it specializes to the inflation among the usual Grothendieck rings studied by Brito--Chari [J.\ Reine Angew.\ Math.\ 804, 2023].
We show that our homomorphisms respect the canonical bases formed by the simple $(q,t)$-characters, which in particular verifies a conjecture of Brito--Chari in loc.\ cit. 
We also discuss a categorification of our homomorphisms using the quiver Hecke algebras of type $\mathrm{A}_\infty$.  
\end{abstract}

\tableofcontents

\section{Introduction}

The monoidal category $\Cc$ of finite-dimensional representations of quantum loop algebra associated with a complex simple Lie algebra  has been studied for over 30 years from various perspectives.
As in the classical representation theory of simple Lie algebras, we have a highest weight type classification of simple representations in $\Cc$ due to Chari--Pressley~\cite{CP95}.
However, beyond such a classification result, the situation is much more intricate.
For instance, there is no known closed formula for the $q$-characters of general simple representations in $\Cc$.

The quantum Grothendieck ring $K_t(\Cc)$, introduced by Nakajima~\cite{Nak} and Varagnolo--Vasserot~\cite{VV03} for type $\mathrm{ADE}$, and by Hernandez~\cite{Her} for general type, is a one-parameter deformation of the Grothendieck ring $K(\Cc)$.
It carries a canonical basis formed by the $t$-analog of $q$-characters (or $(q,t)$-characters for short) of simple representations, constructed through the Kazhdan--Lusztig type algorithm.  
For type $\mathrm{ADE}$, Nakajima~\cite{Nak} applied his theory of quiver varieties to prove an analog of Kazhdan--Lusztig conjecture in this setting, that is, the simple $(q,t)$-characters specialize to the simple $q$-characters in the classical limit $t \to 1$.
Thus, it gives an explicit algorithm to compute the simple $q$-characters of type $\mathrm{ADE}$ in a uniform way.

For the other type $\mathrm{BCFG}$,
Hernandez, Oh, Oya, and the present author  \cite{FHOO, FHOO2} verified the analog of Kazhdan--Lusztig conjecture for some classes of simple representations, including all the simple representations of type $\mathrm{B}$. 
A crucial step in their proof is the construction of a collection of isomorphisms among the quantum Grothendieck rings of different Dynkin types, respecting the canonical bases.
For example, it includes some isomorphisms between the quantum Grothendieck rings of type $\mathrm{B}_n$ and type $\mathrm{A}_{2n-1}$.

This result partly motivates us to investigate further relationship among the quantum Grothendieck rings. 
In this paper, we restrict our attention to type $\mathrm{A}$ and construct a collection of injective homomorphisms among the quantum Grothendieck rings of different ranks. 
(Of course, by composing them with the isomorphisms  between type $\mathrm{B}_{n}$ and type $\mathrm{A}_{2n-1}$ in \cite{FHOO, FHOO2}, we can extend our collection to a collection of injective homomorphisms among type $\mathrm{AB}$.)
To be more precise, let $\Cc_n$ be a monoidal skeleton of the category $\Cc$ (= the subcategory $\Cc_\Z$ in the sense of Hernandez--Leclerc~\cite{HL10}) for the quantum loop algebra of $\mathfrak{sl}_n$, and $K_t(\Cc_n)$ its quantum Grothendieck ring. 
Let $[1,n] \seq \{1,2,\ldots, n \}$ be the integer interval between $1$ and $n$, 
and $I_n \seq [1,n-1]$ an index set of simple roots of $\mathfrak{sl}_n$. 
The main result of this paper is the following.

\begin{Thm}[= Theorem \ref{Thm:main}] \label{Thm:main_intro}
Let $n, \tn$ be two positive integers with $1 < n < \tn$.
For any choice of height functions $\xi \colon I_n \to \Z$, $\txi \colon I_\tn \to \Z$ (cf.~\S\ref{Ssec:pres}) and a (strictly) increasing function $\nu \colon [1,n] \to [1,\tn]$, we have an injective homomorphism 
$\Psi_{\txi, \nu, \xi} \colon K_t(\Cc_n) \to K_t(\Cc_\tn)$ respecting the canonical bases.
\end{Thm}

We call the homomorphisms $\Psi_{\txi, \nu, \xi}$ the \textit{quantum inflations}, as the classical limit of $\Psi_{\txi, \nu, \xi}$ for a special choice of $(\xi, \txi, \nu)$ coincides with the inflation from $K(\Cc_n)$ to $K(\Cc_\tn)$ studied by Brito--Chari~\cite[\S3]{BC}.
Our construction of $\Psi_{\txi, \nu, \xi}$ relies on the presentation of a localization of $K_t(\Cc_n)$ (resp.\ $K_t(\Cc_\tn)$) due to Hernandez--Leclerc~\cite{HL} as the bosonic extension of  the half of the quantized enveloping algebra $U_t(\mathfrak{sl}_n)$ (resp.\ $U_t(\mathfrak{sl}_\tn)$), which depends on the choice of height function $\xi$ (resp.\ $\txi$).
The choice of increasing function $\nu$ gives rise to an injective homomorphism between the bosonic extensions, which is an analog of the homomorphism $\mathfrak{sl}_n \to \mathfrak{sl}_\tn$ of Lie algebras sending the $(i,j)$-matrix unit $E_{i,j} \in \mathfrak{sl}_n$, $i \neq j$, to $E_{\nu(i),\nu(j)} \in \mathfrak{sl}_\tn$.

Combined with the aforementioned result of Nakajima~\cite{Nak}, Theorem~\ref{Thm:main_intro} implies that the classical limit of $\Psi_{\txi, \nu, \xi}$ respects the simple classes.
In particular, we find that Brito--Chari's inflation respects the simple classes, which gives a proof of \cite[Conjecture 3.2]{BC}. 

In the last section of this paper, we briefly discuss a categorification of our quantum inflations.
By Kang--Kashiwara--Kim~\cite{KKK}, the quantum Grothendieck ring $K_t(\Cc_n)$ is known to be isomorphic to the Grothendieck ring $K(\Tt_n)$ of a certain localization $\Tt_n$ of the category of finite-dimensional graded modules over the quiver Hecke algebras of type $\mathrm{A}_\infty$.
In the last section,  we observe the following.

\begin{Thm}[= Theorem~\ref{Thm:categorification}]
For any choice of $(\xi, \txi, \nu)$, there is a graded exact monoidal functor $F_{\txi, \nu, \xi} \colon \Tt_n \to \Tt_\tn$ categorifying $\Psi_{\txi, \nu, \xi}$.
\end{Thm}

\subsection*{Organization}
This paper is organized as follows.
We recall some fundamental facts on $\Cc_n$ and Brito--Chari's inflation in \S\ref{Sec:prel}.
We give a brief review on the quantum Grothendieck ring $K_t(\Cc_n)$ and its presentation in \S\ref{Sec:qGroth}. 
The main part of this paper is \S\ref{Sec:main}, where we define the quantum inflations in \S\ref{Ssec:const}, prove their compatibility with the canonical bases in \S\ref{Ssec:main}, and verify Brito--Chari's conjecture in \S\ref{Ssec:compare}.
We discuss the categorification of the quantum Grothendieck ring $K_t(\Cc_n)$ and the quantum inflations in the final \S\ref{Sec:caton}.

\section{Preliminaries} \label{Sec:prel}
In this section, after fixing our notation in \S\ref
{Ssec:notation}, we quickly review some basic facts on the representation theory of the quantum loop algebra of type $\mathrm{A}$ in \S\ref{Ssec:basic}.
Then, we recall the definition of the inflations introduced by Brito--Chari~\cite{BC} and their conjecture in \S\ref{Ssec:BC}. 

\subsection{Notation} \label{Ssec:notation}
Throughout the paper, we use the following conventions.
For a mathematical statement $P$, we set $\delta(P)$ to be $1$ or $0$ according that $P$ is true or false.
As a special case, we set $\delta_{i,j} \seq \delta(i=j)$.
We denote by $\Z, \Q$, and $\C$ the sets of integers, rational numbers, and complex numbers respectively.
For $a,b \in \Z$ with $a \le b$, we set $[a,b] \seq \{ k \in \Z \mid a \le k \le b\}$.  
Given $n \in \Z$, we write $a \equiv_n b$ when $a -b \in n\Z$.

Fix $n \in \Z_{> 1}$. 
Let $P_n=\bigoplus_{i \in [1,n]} \Z \ep_i$ be a free abelian group of rank $n$ equipped with a standard bilinear pairing  $(\ep_i,\ep_j) = \delta_{i,j}$.
The symmetric group $\SG_n$ of degree $n$ acts on $P_n$ from the left in the natural way.
For $a,b \in [1,n]$ with $a < b$, we put $\alpha_{a,b} \seq \ep_a - \ep_b \in P_n$. 
The set $R_n^+ \seq \{ \alpha_{a,b} \mid a,b \in [1,n], a < b \}$ is identical to the set of positive roots of type $\mathrm{A}_{n-1}$.
Let $I_n \seq [1,n-1].$
For each $i \in I_n$, let $\alpha_i \seq \alpha_{i,i+1} \in P_n$ be the $i$-th simple root and $s_i \in \SG_n$ the $i$-th simple reflection (the transposition of $i$ and $i+1$). 
Let $c_{i,j} \seq (\alpha_i, \alpha_j)$ be the Cartan integer for $i,j \in I_n$. 

We consider the set 
\[ \hI_n \seq \{ (i,p) \in I_n \times \Z \mid p \equiv_2 i \}\]
and the ring of Laurent polynomials 
\[ \cY_n \seq \Z[Y_{i,p}^{\pm 1} \mid (i,p) \in \hI_n].\]
Let $\cM_n^+$ (resp.\ $\cM_n$) be the subset of $\cY_n$ consisting of all the monomials (resp.\ Laurent monomials) in the variables $Y_{i.p}$, $(i,p) \in \hI_n$.
It is an abelian monoid (resp.\ group) by the multiplication.
We refer to an element $m \in \cM_n^+$ as a \textit{dominant monomial}.

\subsection{Quantum loop algebras and $q$-characters}
\label{Ssec:basic} 
Let $U_q(\Lsl_{n})$ be the quantum loop algebra, associated with the complex simple Lie algebra $\fsl_n$ of type $\mathrm{A}_{n-1}$.
This is a Hopf algebra over $\C$.
We refer to \cite[Chapter 12]{CP} for its precise definition.
In this paper, we assume that $q$ is not a root of unity.

Recall that simple finite-dimensional modules of type $\mathbf{1}$ over $U_q(\Lsl_{n})$ are parametrized by the multiplicative monoid $(1 + z\C[z])^{I_n}$ of Drinfeld polynomials up to isomorphism by the fundamental result of Chari--Pressley~\cite{CP, CP95}.  
For $\pi \in (1 + z\C[z])^{I_n}$, we denote by $L(\pi)$ the corresponding simple $U_q(\Lsl_{n})$-module.
We regard the monoid $\cM_n^+$ of dominant monomials as a submonoid of $(1 + z\C[z])^{I_n}$ through the correspondence
\[ Y_{i,p} \mapsto (1-\delta_{i,j}q^pz)_{j \in I_n}.\]
Then we consider the category  $\Cc_n$ of finite-dimensional $U_q(\Lsl_{n})$-modules whose composition factors are isomorphic to $L(m)$'s for $m \in \cM_n^+$.
This category is the same as Hernandez--Leclerc's category $\Cc_\Z$ for $U_q(L\mathfrak{sl}_n)$ in \cite{HL10}.
(Note that the symbol $\Cc_n$ has a different meaning in \cite{HL10}.)
The category $\Cc_n$ is closed under taking tensor products and duals. 
Therefore, it carries a natural structure of rigid monoidal abelian category.  

Let $K(\Cc_n)$ be the Grothendieck ring of the category $\Cc_n$.
It comes with the free $\Z$-basis $\{ [L(m)] \mid m \in \cM_n^+\}$ of simple isomorphism classes.
The $q$-character map, in the sense of Frenkel--Reshetikhin \cite{FR}, induces an injective ring homomorphism
\[ \chi_q \colon K(\Cc_n) \to \cY_n.\]
By Frenkel--Mukhin \cite{FM}, the image of $\chi_q$ coincides with the intersection
\begin{equation} \label{eq:capK}
\bigcap_{i \in I_n} \Z[Y_{i,p}(1+A_{i,p+1}^{-1}), Y_{j,s}^{\pm 1}\mid j \neq i, p \equiv_2 i, s \equiv_2 j],
\end{equation}
where $A_{i,p} \seq Y_{i,p-1}Y_{i,p+1}Y^{-1}_{i-1,p}Y^{-1}_{i+1,p}$ with $Y_{0,p} = Y_{n,p} = 1$ by convention.
In what follows, we often identify $K(\Cc_n)$ with the image of $\chi_q$.

A simple module of the form $L(Y_{i,p})$ for some $(i,p) \in \hI_n$ is called a \textit{fundamental module}.
As a commutative ring, $K(\Cc_n)$ is freely generated by the classes of fundamental modules.

\subsection{Brito--Chari's inflation} \label{Ssec:BC}
Let $n, \tn \in \Z_{>1}$ satisfying $\tn/n \in \Z$.
In \cite{BC}, Brito--Chari introduced a injective ring homomorphism
\[ \Psi_{\tn,n} \colon K(\Cc_{n}) \to K(\Cc_\tn)\]
given by the assignment $\chi_q(L(Y_{i,p})) \mapsto \chi_q(L(Y_{\tn i/n, \tn p/n}))$ for all $(i,p) \in \hI_{n}$.
The homomorphism $\Psi_{\tn,n}$ is called the \textit{inflation}.

\begin{Conj}[{Brito--Chari \cite[Conjecture 3.2]{BC}}]
\label{Conj:BC}
The inflation $\Psi_{\tn,n}$ respects the simple isomorphism classes.
More precisely, letting \[\psi_{\tn,n} \colon \cM_{n}^+ \to \cM_{\tn}^+\] be the homomorphism given by $Y_{i,p} \mapsto Y_{\tn i/n,\tn p/n}$ for $(i,p) \in \hI_{n}$, we have
\begin{equation} \label{eq:conj}
\Psi_{\tn,n}([L(m)]) = [L(\psi_{\tn,n}(m))]
\end{equation}
for any $m \in \cM_{n}^+$.  
\end{Conj}

In \cite[Theorem 2]{BC}, Brito--Chari verifies the relation \eqref{eq:conj} for some special classes of simple modules $L(m)$ such as (usual and higher order) Kirillov--Reshetikhin modules, snake modules, and arbitrary simple modules in a level-one subcategory $\Cc_\xi$ associated with an increasing height function $\xi$ as in \cite{HLad}.

Later in \S\ref{Ssec:compare}, we give a proof of Conjecture~\ref{Conj:BC} by constructing an analog of the inflation $\Psi_{\tn,n}$ between the quantum Grothendieck rings and showing that it respects their canonical bases.

\section{Quantum Grothendieck rings of type $\mathrm{A}$} 
\label{Sec:qGroth}
In this section, we recall the definition of the quantum Grothendieck ring of the category $\Cc_n$ and their canonical bases following Nakajima~\cite{Nak} and Hernandez~\cite{Her}. 
We also recall its presentation as the bosonic extension of the half of the quantized enveloping algebra of $\mathfrak{sl}_n$ due to Hernandez--Leclerc~\cite{HL}.
The presentation depends on the choice of a height function on the Dynkin diagram of $\mathfrak{sl}_n$ and plays a key role in our construction of the quantum analog of inflation in \S\ref{Sec:main}. 

\subsection{Definition}
Let $z$ be a formal parameter and $C(z) \seq (\frac{z^{c_{i,j}}-z^{-c_{i,j}}}{z-z^{-1}})_{i,j \in I_n}$ a deformed Cartan matrix. 
For any $i,j \in I_n$, the $(i,j)$-entry $\tC_{i,j}(z)$ of the inverse matrix $C(z)^{-1}$ is a rational function in $z$.
We write its Laurent expansion at $z=0$ as
\[ \tC_{i,j}(z) = \sum_{k \in \Z} \tc_{i,j}(k) z^k.\]
It turns out that all the coefficients $\tc_{i,j}(k)$ are integers.
It is known that we have $\tc_{i,j}(k) = 0$ if $k \le 0$ or $k \equiv_2 i+j$ (see \cite[\S2]{HL}).

Define the skew-symmetric bilinear map $\Nn \colon \cM_n \times \cM_n \to \Z$ of abelian groups by $\Nn(Y_{i,p}, Y_{j,s}) = \Nn_{i,j}(p-s)$, where
\[\Nn_{i,j}(k) \seq \tc_{i,j}(k-1)-\tc_{i,j}(k+1) -\tc_{i,j}(-k-1) +\tc_{i,j}(-k+1),  \]
for any $i,j \in I_n$ and $k \in \Z$.
An explicit formula of $\Nn_{i,j}(k)$ is given in Example \ref{Ex:phi} below.

Let $t$ be another formal parameter with a square root $t^{1/2}$. 
We endow the free $\Z[t^{\pm 1/2}]$-module 
\[\cY_{n,t} \seq \Z[t^{\pm 1/2}] \otimes_\Z \cY_n = \bigoplus_{m \in \cM_n} \Z[t^{\pm 1/2}] m\] with a $\Z[t^{\pm 1/2}]$-bilinear product $*$ by $m* m' \seq t^{\Nn(m,m')/2}mm'$
for $m, m' \in \cM_n$.
The resulting $\Z[t^{\pm 1/2}]$-algebra  $\cY_{n,t} = (\cY_{n,t}, *)$ is a quantum torus, which gives a non-commutative deformation of $\cY_n$. 
The specialization at $t^{1/2} = 1$ yields a ring homomorphism $\ev_{t=1} \colon \cY_{n,t} \to \cY_n$.
In what follows, we often identify an element $y \in \cY_n$ with $1 \otimes y \in \cY_{n,t}$. 
It gives an inclusion $\cY_n = 1 \otimes \cY_n \subset \cY_{n,t}$ of a subgroup (not of a subalgebra).

We have the natural anti-involution $y \mapsto \ol{y}$ of $\cY_{n,t}$ given by $\ol{t^{1/2}} = t^{-1/2}$ and $\ol{m} = m$ for all $m \in \cM_n$. 
We call it the \textit{bar-involution}.
Each element of the subgroup $\cY_n \subset \cY_{n,t}$ is fixed by this bar-involution.

For each $i \in I_n$, let $K_{t,i}$ be the $\Z[t^{\pm 1/2}]$-subalgebra of the quantum torus $\cY_{n,t}$ generated by the set 
\[ \{Y_{i,p}(1+A_{i,p+1}^{-1}) \mid p \equiv_2 i \} \cup \{ Y_{j,s}^{\pm 1}\mid j \neq i, s \equiv_2 j\}.\]
Following Hernandez \cite{Her}, we define the \textit{quantum Grothendieck ring} $K_t(\Cc_n)$ to be the  intersection of them:
\[ K_t(\Cc_n) \seq \bigcap_{i \in I_n} K_{t,i}. \]
Compare to \eqref{eq:capK}.
By construction, $K_t(\Cc_n)$ is a $\Z[t^{\pm 1/2}]$-subalgebra of $\cY_{n,t}$.
By \cite[Corollary 3.6]{Nak} or \cite[Theorem 6.2]{Her}, we have 
\[\ev_{t=1}(K_t(\Cc_n)) = \chi_q(K(\Cc_n)) \simeq K(\Cc_n).\]

\subsection{Canonical bases}
By \cite[Proposition 9.3 \& Theorem 9.4]{HO}, the $q$-character of a fundamental module $\chi_q(L(Y_{i,p}))$ belongs to $K_t(\Cc_n)$. 
This is a special feature of type $\mathrm{A}$ (and $\mathrm{B}$).

\begin{Lem} \label{Lem:comm}
For $(i,p), (j,s) \in \hI_n$, we have 
\[\chi_{q}(L(Y_{i,p})) * \chi_{q}(L(Y_{j,s})) = \chi_q(L(Y_{j,s})) * \chi_q(L(Y_{i,p})) \quad \text{if $|p-s| < |i-j| + \delta_{i,j}$.} \]
\end{Lem}
\begin{proof}
This follows from \cite[Lemma 9.11(3)]{FHOO}.
See also Remark~\ref{Rem:Ft} below.
\end{proof}

Write a dominant monomial $m \in \cM_n^+$ as $m= Y_{i_1,p_1} \cdots Y_{i_d, p_d}$ with $p_1 \ge \cdots \ge p_d$.
Then, we set
\begin{equation} \label{eq:E} 
E_t(m) \seq t^{-\sum_{1 \le k < l \le d}\Nn(Y_{i_k,p_k}, Y_{i_l, p_l})/2} \chi_q(L(Y_{i_1,p_1})) * \cdots * \chi_{q}(L(Y_{i_d,p_d})). 
\end{equation}
Since $\chi_q(L(Y_{i,p}))$ and $\chi_q(L(Y_{j,p}))$ mutually commute with respect to $*$ by Lemma~\ref{Lem:comm}, and $\Nn_{i,j}(0)=0$ for any $i,j \in I_n$, the element $E_t(m)$ is well-defined (independent of the ordering of the factors of $m$).   
The set $\{ E_t(m) \mid m \in \cM_n^+\}$ gives a free $\Z[t^{\pm 1/2}]$-basis of $K_t(\Cc_n)$, called the \textit{standard basis}.

\begin{Rem} \label{Rem:Ft}
In \cite{Her}, the standard basis element $E_t(m)$ is defined in terms of the element $F_t(Y_{i,p})$ such that $Y_{i,p}$ is the unique dominant monomial occurring in $F_t(Y_{i,p})$, instead of $\chi_q(L(Y_{i,p}))$.
Our definition of $E_t(m)$ is equivalent to the one in loc.\ cit.\ as $F_t(Y_{i,p}) = \chi_q(L(Y_{i,p}))$ holds in type $\mathrm{A}$.
Note that the last equality is not true in general.
\end{Rem}

\begin{Thm}[{Nakajima \cite{Nak}, Hernandez \cite{Her}, see also \cite[Remark 7.8]{HO}}] \label{Thm:canon}
For each dominant monomial $m \in \cM_n^+$, there exists a unique element $\chi_{q,t}(L(m))$ of $K_t(\Cc_n)$ satisfying 
\[ \ol{\chi_{q,t}(L(m))} = \chi_{q,t}(L(m)) \quad \text{and} \quad \chi_{q,t}(L(m)) - E_t(m) \in \sum_{m' \in \cM_n^+} t^{-1}\Z[t^{-1}]E_t(m').\]
Moreover, the set $\{ \chi_{q,t}(L(m)) \mid m \in \cM_n^+ \}$ forms a free $\Z[t^{\pm 1/2}]$-basis of $K_t(\Cc_n)$, called the canonical basis of $K_t(\Cc_n)$.
\end{Thm}
\begin{Rem}
The transition matrix from the standard basis to the canonical basis turns out to be unitriangular with respect to the so-called Nakajima partial ordering of dominant monomials. 
In fact, the original defining condition of the canonical basis in \cite{Nak, Her} requires this unitriangularity as well. 
The weaker condition in Theorem \ref{Thm:canon} is enough to characterize the basis, as explained in \cite[Remark 7.8]{HO}.
\end{Rem}

The element $\chi_{q,t}(L(m))$ is called the \textit{$t$-analog of $q$-character} (or the \textit{$(q,t)$-character} for short) of the simple module $L(m)$.
For a fundamental module $L(Y_{i,p})$,  we have 
$\chi_{q,t}(L(Y_{i,p})) = \chi_q(L(Y_{i,p})).$

The following fundamental result was proved by Nakajima using the geometry of quiver varieties.
\begin{Thm}[{Nakajima \cite{Nak}}] \label{Thm:Nak}
For each $m \in \cM_n^+$, we have
\[\ev_{t=1} \chi_{q,t}(L(m)) = \chi_q (L(m)).\] 
\end{Thm}

\subsection{Presentations} \label{Ssec:pres}
We say that a function $\xi \colon I_n \to \Z$ is a \textit{height function} if it satisfies $\xi(1) \equiv_2 1$ and $|\xi(i) - \xi(i+1)| = 1$ for all $i \in [1,n-2]$. 
A height function $\xi$ defines a Coxeter element $\tau_\xi \in \SG_n$ by
$\tau_\xi \seq s_{i_1} s_{i_2} \cdots s_{i_{n-1}}$,
where we chose a total ordering $I_n = \{ i_1, i_2, \ldots , i_{n-1}\}$ satisfying $\xi(i_1) \le \xi(i_2) \le \cdots \le \xi(i_{n-1})$.
The element $\tau_\xi$ does not depend on the choice of such a total ordering. 
Following \cite[\S2]{HL}, we recursively define the bijection 
\[\phi_\xi \colon \hI_n \xrightarrow{\sim} R_n^+ \times \Z \]
by the following two requirements:
\begin{enumerate}
\item we have $\phi_\xi(i,\xi(i)) = (\sum_{j \in I_n, \xi_j - \xi_i = |j-i|} \alpha_j, 0)$ for any $i \in I_n$,
\item if $\phi_{\xi}(i,p) = (\alpha,k)$, we have
\[\phi_\xi(i,p\pm2) = \begin{cases}
(\tau_\xi^{\pm 1}(\alpha), k) & \text{if $\tau_\xi^{\pm 1}(\alpha) \in R^+_n$}, \\
(-\tau_\xi^{\pm 1}(\alpha), k\pm 1) & \text{if $\tau_\xi^{\pm 1}(\alpha) \not\in R^+_n$}.
\end{cases}
\]
\end{enumerate}

\begin{Rem}
Our $\tau_\xi$ is the same as $\tau^{-1}$ in  \cite[\S2]{HL}. 
Our bijection $\phi_\xi$ is the same as the bijection $\varphi$ in loc.\ cit.
\end{Rem}

For future use, we recall a relation between the bilinear form $\Nn$ and the bijection $\phi_\xi$.
\begin{Lem} \label{Lem:FO}
Let $\xi \colon I_n \to \Z$ be a height function.
For any $(i,p), (j,s) \in \hI_n$ with $(i,p) \neq (j,s)$, we have
\[ \Nn(Y_{i,p}, Y_{j,s}) = (-1)^{k+l+\delta(p\ge s)}(\alpha, \beta),\]
where $(\alpha,k) = \phi_\xi(i,p)$ and $(\beta, l) = \phi_\xi(j,s)$.
Moreover, we have 
\[\Nn(Y_{i,p}, Y_{j,s}) = 0 \quad \text{if $p \ge s$ and $k < l$.} \]
\end{Lem}
\begin{proof}
This follows from \cite[Proposition 3.2]{HL} and \cite[Lemma 9.11(2)]{FHOO}. 
\end{proof}

\begin{Ex} \label{Ex:phi}
The most important case for the purpose of this paper is when our height function $\xi$ is \textit{increasing}, namely when it satisfies 
\begin{equation}
\xi(i) = i + 2c  \quad \text{for any $i \in I_n$}
\end{equation}
with a constant $c \in \Z$.
In this case, the Coxeter element $\tau_\xi$ is $s_1 s_2 \cdots s_{n-1}$, which acts on $P_n$ by $\ep_i \mapsto \ep_{i+1}$ for all $i \in [1,n]$. 
Here $\ep_{n+1} \seq \ep_1$ by convention. 
The bijection $\phi_\xi$ is easy to compute as follows.
For an integer $p \in \Z$, we write it as
\begin{equation} \label{eq:euclid}
p = k_n(p)n + r_n(p) \quad \text{with $k_n(p) \in \Z$ and $r_n(p) \in [1,n]$}.
\end{equation}
Then, we have 
\[ \phi_\xi(i,\xi(i)+2p) = \begin{cases}
(\alpha_{r_n(p), r_n(p+i)}, 2k_n(p)+1) & \text{if $r_n(p) < r_n(p+i)$}, \\
(\alpha_{r_n(p+i), r_n(p)}, 2k_n(p)+2)& \text{if $r_n(p) > r_n(p+i)$},
\end{cases}\]
for any $i \in I_n$ and $p \in \Z$. 
Therefore, the inverse bijection $\phi_\xi^{-1}$ is given as
\begin{equation} \label{eq:phi-1}
\begin{cases}
\phi_\xi^{-1}(\alpha_{a,b}, 2k+1) = (b-a,\xi(b-a) +2(kn+a)), \\
\phi_\xi^{-1}(\alpha_{a,b}, 2k+2) = (n+a-b, \xi(n+a-b) + 2(kn+b)).
\end{cases}
\end{equation}
With these computations, Lemma~\ref{Lem:FO} yields an explicit formula of $\Nn_{i,j}(k)$:
for any $i,j \in I_n$ and $k \ge \delta_{i,j}$, we have
\[ \Nn_{i,j}(k) = \delta(k \equiv_{2n} i+j) - \delta(k \equiv_{2n} i-j) - \delta(k \equiv_{2n} -i+j) + \delta(k \equiv_{2n} -i-j).\]
\end{Ex}

Now, we recall a presentation of the quantum Grothendieck ring $K_t(\Cc_n)$, which depends on the choice of height function.
\begin{Def}[Bosonic extension]
Let $\hA_n$ be a $\Q(t^{1/2})$-algebra presented by the generators
$\{e_{i,k} \mid i \in I_n, k\in \Z \}$ with the following relations:
\begin{enumerate}
\renewcommand{\theenumi}{\rm R\arabic{enumi}}
\item \label{R1} 
For any $k \in \Z$ and any $i,j \in I_n$, we have 
\[ \begin{cases} 
e_{i,k}^2 e_{j,k} - (t+t^{-1})e_{i,k} e_{j,k} e_{i,k} + e_{j,k} e_{i,k}^2 = 0 & \text{if $|i-j|=1$}, \\
e_{i,k} e_{j,k} - e_{j,k} e_{i,k} = 0& \text{if $|i-j|>1$}.
\end{cases} \]
\item \label{R2}
For any $k,k' \in \Z$ with $k <k'$ and $i,j \in I_n$, we have 
\[e_{i,k}e_{j,k'} = t^{(-1)^{k+k'}c_{i,j}} e_{j,k'}e_{k,i} + (1-t^{-2})\delta_{(i,k), (j,k'-1)}.\]
\end{enumerate}
The algebra $\hA_n$ is often called the \textit{bosonic extension} of the positive half $U_n^+$ of the quantized enveloping algebra of $\mathfrak{sl}_{n}$ (see \S \ref{Ssec:const}). 
\end{Def}

Put 
\[K_t(\Cc_n)_{loc} \seq K_t(\Cc_n) \otimes_{\Z[t^{\pm 1/2}]} \Q(t^{1/2}).\]

\begin{Thm}[{Hernandez--Leclerc~\cite{HL}, F.--Hernandez--Oh--Oya~\cite{FHOO}}]
\label{Thm:Psi}
\hfil
\begin{enumerate}
\item \label{Thm:Psi1} 
Let $\xi \colon I_n \to \Z$ be a  height function. 
There is an isomorphism 
\[ \Psi_\xi \colon \hA_n \xrightarrow{\sim} K_t(\Cc_n)_{loc} \]
of $\Q(t^{1/2})$-algebras satisfying 
\[\Psi_\xi(e_{j,k}) = \chi_{q,t}(L(Y_{i,p})) \quad \text{if $\phi_\xi(i,p) = (\alpha_j, k)$} \]
for any $(j,k) \in I_n \times \Z$. 
\item \label{Thm:Psi2}
Let $\xi' \colon I_n \to \Z$ be another height function. 
The automorphism $\Psi_{\xi, \xi'} \seq \Psi_{\xi} \circ \Psi_{\xi'}^{-1}$ of $K_t(\Cc_n)_{loc}$ induces a permutation of the simple $(q,t)$-characters. 
In other words, there is a certain permutation $\psi_{\xi,\xi'}$ of the set $\cM_n^+$ such that we have
\[ \Psi_{\xi, \xi'}(\chi_{q,t}(L(m))) = \chi_{q,t}(L(\psi_{\xi,\xi'}(m)))\]
for any $m \in \cM_n^+$.
In particular, $\Psi_{\xi,\xi'}$ restricts to a $\Z[t^{\pm 1/2}]$-algebra automorphism of $K_t(\Cc_n)$.
\end{enumerate}
\end{Thm}

\begin{Rem} \label{Rem:psi}
In general, the permutation $\psi_{\xi, \xi'}$ of $\cM_n^+$ in Theorem~\ref{Thm:Psi}~\eqref{Thm:Psi2} is only piecewise linear, not a monoid automorphism. 
By construction, it satisfies $\psi_{\xi,\xi'}(Y_{i,p}) = Y_{j,s}$ if $\phi_{\xi'}(i,p) = \phi_\xi(j,s) = (\alpha, k)$ with $\alpha$ being a simple root.
But it is not always true if $\alpha$ is a non-simple root.
\end{Rem}

\begin{Rem} \label{Rem:FHOO2}
If $\xi - \xi' = 2c$ for some constant $c \in \Z$, the automorphism $\Psi_{\xi,\xi'}$ coincides with the spectral parameter shift $Y_{i,p} \mapsto Y_{i,p+2c}$. 
In general, $\Psi_{\xi,\xi'}$ can be written as a composition of the braid symmetries $\sigma_{i}$, $i \in I_n$, introduced in \cite{KKOPbr}, which we use in the proof of Proposition~\ref{Prop:Upsilon} below, and a spectral parameter shift. 
See \cite[the paragraph before Proposition 6.2]{FHOO2}.
\end{Rem}

\section{Quantum inflations} \label{Sec:main}
In this section, we introduce a collection of injective homomorphisms from $K_t(\Cc_n)$ to $K_t(\Cc_\tn)$ with $n < \tn$, which we call the quantum inflations as it includes a quantum analog of Brito--Chari's inflation in \S\ref{Ssec:BC} as a special case.
Our main theorem (= Theorem \ref{Thm:main}) in \S\ref{Ssec:main} asserts that they respect the canonical bases, and hence their classical limits respect the simple classes.    
In particular, we obtain in \S\ref{Ssec:compare} a proof of Conjecture~\ref{Conj:BC}. 

\subsection{Construction} \label{Ssec:const}
Let $n \in \Z_{>1}$ and $U_n^+$ the positive half of the quantized enveloping algebra of $\mathfrak{sl}_n$ over $\Q(t^{1/2})$.
Namely, $U_n^+$ is the $\Q(t^{1/2})$-algebra presented by the generators $\{ e_i \mid i \in I_n\}$ with the quantum Serre relations:
\[ \begin{cases} 
e_{i}^2 e_{j} - (t+t^{-1})e_{i} e_{j} e_{i} + e_{j} e_{i}^2 = 0 & \text{if $|i-j|=1$}, \\
e_{i} e_{j} - e_{j} e_{i} = 0& \text{if $|i-j|>1$}.
\end{cases} \]
For each $k \in \Z$, the assignment $e_i \mapsto e_{i,k}$ gives rise to a $\Q(t^{1/2})$-algebra homomorphism 
\[\iota_k \colon U_n^+ \to \hA_n.\] 

In the sequel, we use the following $t$-commutator notation:
\[ [x,y]_{t^{1/2}} \seq t^{1/2}xy - t^{-1/2}yx.\]
For any $a,b \in [1,n]$ with $a < b$, we define the element $e(\alpha_{a,b})$ of $U_n^+$ by
\begin{equation} \label{eq:ealpha}
e(\alpha_{a,b}) \seq \frac{[e_a,[e_{a+1}, \cdots[e_{b-2}, e_{b-1}]_{t^{1/2}} \cdots ]_{t^{1/2}}]_{t^{1/2}}}{(t-t^{-1})^{b-a-1}}.
\end{equation}
We have $e(\alpha_i) = e_i$ for any $i \in I_n$ by definition. 
Note that $e(\alpha_{a,b})$ and $e(\alpha_{c,d})$ commute if $b < c$.

\begin{Lem} \label{Lem:tcomm}
For any $n \in \Z_{>1}$ and $a,b,c \in [1,n]$ with $a<b<c$, we have
\[ e(\alpha_{a,c}) = \frac{[e(\alpha_{a,b}), e(\alpha_{b,c})]_{t^{1/2}}}{t-t^{-1}}.\]
\end{Lem}
\begin{proof}
For $x,y,z \in U_n^+$, we have $[[x,y]_{t^{1/2}}, z]_{t^{1/2}} = [x,[y,z]_{t^{1/2}}]_{t^{1/2}}$ if $xz = zx$.
With this remark in mind, the assertion follows by induction on $c-a$.
\end{proof}

Let $n, \tn \in \Z_{>1}$ with $n < \tn$, and $\nu \colon [1,n] \to [1,\tn]$ a (strictly) increasing function.
Consider an assignment from $\{e_{i,k} \mid i \in I_n, k \in \Z\} \subset \hA_n$ to $\hA_\tn$ given by
\begin{equation} \label{eq:anu}
e_{i,k} \mapsto \iota_k(e(\alpha_{\nu(i), \nu(i+1)})) \in \hA_\tn \quad \text{for each $i \in I_n$ and $k \in \Z$}.
\end{equation}
\begin{Ex}\label{Ex:2to3}
For example, when $n=2$ and $\tn =3$, there are three increasing functions $\nu \colon [1,2] \to [1,3]$. The assignment \eqref{eq:anu} is written as 
\[ 
e_{1,k} \mapsto \begin{cases}
e_{1,k} & \text{if $(\nu(1), \nu(2)) = (1,2)$}, \\
(t-t^{-1})^{-1}[e_{1,k},e_{2,k}]_{t^{1/2}} &  \text{if $(\nu(1), \nu(2)) = (1,3)$}, \\
e_{2,k} & \text{if $(\nu(1), \nu(2)) = (2,3)$}, 
\end{cases}
\] 
for each $k \in \Z$.
\end{Ex}

\begin{Prop} \label{Prop:Upsilon}
For any increasing function $\nu \colon [1,n] \to [1,\tn]$, the assignment \eqref{eq:anu} gives rise to a $\Q(t^{1/2})$-algebra homomorphism
\[ \Upsilon_\nu \colon \hA_n \to \hA_\tn. \]
\end{Prop}

\begin{proof}
For the sake of simplicity, we put $\te_{i,k} \seq \iota_k(e(\alpha_{\nu(i), \nu(i+1)}))$ in this proof.
We have to verify that the collection $\{ \te_{i,k} \mid i \in I_n, k \in \Z \} \subset \hA_\tn$ satisfies the defining relations \eqref{R1} \& \eqref{R2} of the algebra $\hA_n$.
We proceed by induction on $\tn - n \in \Z_{>0}$.

First, we consider the case when $\tn = n+1$.
Note that, for any increasing function $\nu \colon [1,n] \to [1,n+1]$, there is a unique $l \in [0,n]$ such that we have $\nu(i) = i + \delta(i > l)$ for all $i \in I_n$.
Then, we have
\begin{equation} \label{eq:te}
\te_{i,k} = \begin{cases}
e_{i,k} & \text{if $i < l$}, \\
(t-t^{-1})^{-1}[e_{l,k}, e_{l+1,k}]_{t^{1/2}} & \text{if $i = l$}, \\
e_{i+1,k} & \text{if $i > l$}.
\end{cases}
\end{equation}
The desired relations \eqref{R1} \& \eqref{R2} for $\te_{i,k}$ and $\te_{j,k'}$ with $i,j \in I_n \setminus \{l\}$ follows immediately from this formula.
For the remaining relations, we can check them by using the braid group symmetry generated by a collection $\{ \sigma_l^{\pm 1} \mid l \in I_{n+1} \}$ of $\Q(t^{1/2})$-algebra automorphisms of $\hA_{n+1}$ introduced by Kashiwara--Kim--Oh--Park in \cite{KKOPbr}.
Recall that these automorphisms are given by the formula
\begin{align*}
\sigma_l(e_{i,k}) &= \begin{cases}
(t-t^{-1})^{-1}[e_{i,k}, e_{l,k}]_{t^{1/2}} & \text{if $|i-l|=1$}, \\
e_{i,k+\delta_{i,l}} & \text{otherwise}, 
\end{cases} \\
\sigma^{-1}_l(e_{i,k}) &= \begin{cases}
(t-t^{-1})^{-1}[e_{l,k}, e_{i,k}]_{t^{1/2}} & \text{if $|i-l|=1$}, \\
e_{i,k-\delta_{i,l}} & \text{otherwise}.
\end{cases}
\end{align*}
See \cite[Theorem 2.3]{KKOPbr} (also \cite[Theorem 3.1]{KKOPbr2} for more general statement with a detailed proof).
Comparing the formula with \eqref{eq:te}, we have
\[ \te_{i,k} = \begin{cases}
\sigma_{l+1} (e_{i,k}) & \text{if $i \le l$}, \\
\sigma_{l}^{-1}(e_{i+1,k}) & \text{if $i \ge l$}.
\end{cases} \]
Using this, the relations \eqref{R1} \& \eqref{R2} for $\te_{i,k}$ and $\te_{j,k'}$ with $i = l$ or $j=l$ can be checked immediately.  
Thus, $\Upsilon_{\nu}$ is well-defined when $\tn = n+1$.

For the induction step, we note that an increasing  function $\nu \colon [1,n] \to [1,\tn]$ with $\tn - n >1$ can be factorized as $\nu = \nu'' \circ \nu'$, where $\nu' \colon [1,n] \to [1,n']$ and $\nu'' \colon [1,n'] \to [1,\tn]$ are some increasing functions with $n' \in [n+1,\tn-1]$.
By induction hypothesis, the $\Q(t^{1/2})$-algebra homomorphisms $\Upsilon_{\nu'}$ and $\Upsilon_{\nu''}$ are well-defined.
We apply Lemma~\ref{Lem:tcomm} repeatedly to find that
\[ \te_{i,k} =  (\Upsilon_{\nu''} \circ \Upsilon_{\nu'})(e_{i,k})\]
holds for any $i \in I_n$ and $k \in \Z$.
This implies that the collection $\{ \te_{i,k} \mid i \in I_n, k \in \Z \} $ satisfies the desired relations \eqref{R1} \& \eqref{R2}.
Thus, the homomorphism $\Upsilon_\nu$ is also well-defined and $\Upsilon_\nu = \Upsilon_{\nu''} \circ \Upsilon_{\nu'}$ holds. 
\end{proof}

\begin{Def}[Quantum inflations] \label{Def:infl}
Let $n, \tn \in \Z_{>1}$.
For any height functions $\xi \colon I_n \to \Z, \txi \colon I_\tn \to \Z$ and any increasing function $\nu \colon [1,n]\to [1,\tn]$, we define the \textit{quantum inflation}
\[\Psi_{\txi, \nu, \xi} \colon K_t(\Cc_n)_{loc} \to K_t(\Cc_\tn)_{loc}\] 
to be the composition $\Psi_{\txi, \nu, \xi} \seq \Psi_{\txi} \circ \Upsilon_\nu \circ \Psi_\xi^{-1}$, where $\Psi_\xi$ and $\Psi_{\txi}$ are the isomorphisms from Theorem~\ref{Thm:Psi}.
\end{Def}

\begin{Rem} \label{Rem:infl} \hfil
\begin{enumerate}
\item When $n=\tn$, the homomorphism $\Upsilon_\nu$ is the identity automorphism of $K_t(\Cc_n)$ and we have $\Psi_{\txi, \nu, \xi} = \Psi_{\txi, \xi}$ in the notation of Theorem~\ref{Thm:Psi}~\eqref{Thm:Psi2}. 
\item Let $n_1, n_2, n_3 \in \Z_{>1}$ with $n_1 \le n_2 \le n_3$.
For any height functions $\xi_i \colon I_{n_i} \to \Z$, $i \in \{1,2,3\}$, and increasing functions $\nu_i \colon [1,n_i] \to [1,n_{i+1}]$, $i \in \{1,2\}$, we have $\Psi_{\xi_3, \nu_2, \xi_2} \circ \Psi_{\xi_2, \nu_1, \xi_1} = \Psi_{\xi_3, \nu_2 \circ \nu_1, \xi_1}.$
This follows from the construction and the equality $\Upsilon_{\nu_2} \circ \Upsilon_{\nu_1} = \Upsilon_{\nu_2 \circ \nu_1}$ as we saw in the proof of Proposition~\ref{Prop:Upsilon}.
\end{enumerate}
\end{Rem}

\begin{Rem}
When $\tn = n+1$ with $\nu(1)=1$ and $\nu(n) = \tn$, the homomorphism $\Upsilon_\nu \colon \hA_n \to \hA_{n+1}$ can be thought of the bosonic extension of an edge contraction $U_n^+ \to U_{n+1}^+$ in the sense of Li~\cite{Li}.  
In general, $\Upsilon_\nu$ is the bosonic extension of an iteration of edge contractions. 
\end{Rem}

\subsection{Compatibility with canonical bases} \label{Ssec:main}
We prove that the quantum inflations respect the canonical bases of the quantum Grothendieck rings.
We retain the notation from the previous subsection.
\begin{Thm} \label{Thm:main}
The quantum inflation $\Psi_{\txi, \nu, \xi}$ in Definition~\ref{Def:infl} induces an injective map from the canonical basis of $K_t(\Cc_n)$ into the canonical basis of $K_t(\Cc_\tn)$.
In other words, there is an injective map $\psi_{\txi, \nu, \xi} \colon \cM_n^+ \to \cM_\tn^+$ such that we have
\begin{equation} \label{eq:main}
\Psi_{\txi, \nu. \xi}(\chi_{q,t}(L(m))) = \chi_{q,t}(L(\psi_{\txi,\nu,\xi}(m)))
\end{equation}
for any $m \in \cM_n^+$.
In particular, $\Psi_{\txi, \nu, \xi}$ restricts to an injective $\Z[t^{\pm1/2}]$-algebra homomorphism
\[ \Psi_{\txi, \nu, \xi} \colon K_t(\Cc_n) \hookrightarrow K_t(\Cc_\tn),\]
compatible with the bar-involutions.
\end{Thm}
\begin{proof}
Let $\xi' \colon I_n \to \Z$ and $\txi' \colon I_\tn \to \Z$ be another pair of height functions.
By Remark~\ref{Rem:infl}, we have $\Psi_{\txi, \nu, \xi} = \Psi_{\txi, \txi'} \circ \Psi_{\txi', \nu, \xi'} \circ \Psi_{\xi', \xi}$.
Therefore, as soon as we prove the statement for $\Psi_{\txi', \nu, \xi'}$ with an injective map $\psi_{\txi', \nu, \xi'} \colon \cM_n^+ \to \cM_\tn$, the statement for $\Psi_{\txi, \nu, \xi}$ also follows with the map $\psi_{\txi, \nu, \xi} \seq \psi_{\txi, \txi'} \circ \psi_{\txi', \nu, \xi'} \circ \psi_{\xi', \xi}$, thanks to Theorem~\ref{Thm:Psi}~\eqref{Thm:Psi2}.
Thus, it suffices to consider the special case when both height functions $\xi$ and $\txi$ are increasing, which is treated in Proposition~\ref{Prop:main} below.
\end{proof}

With a given increasing function $\nu \colon [1,n] \to [1,\tn]$, we associate an injective map $\nu_* \colon R_n^+ \to R_\tn^+$ by 
\[ \nu_*(\alpha_{a,b}) \seq \alpha_{\nu(a), \nu(b)} \quad \text{for any $a,b \in [1,n]$ with $a < b$}.\]

\begin{Prop} \label{Prop:main}
Assume that both $\xi \colon I_n \to \Z$ and $\txi \colon I_\tn \to \Z$ are increasing height functions. 
Then, the equality \eqref{eq:main} holds for any $m \in \cM_n^+$ with $\psi_{\txi, \nu, \xi} \colon \cM_n^+ \to \cM_\tn^+$ being the injective monoid homomorphism given by
\[ \psi_{\txi, \nu, \xi}(Y_{i,p}) \seq Y_{j,s} \quad \text{if $(j,s) = (\phi_\txi^{-1} \circ(\nu_* \times \id_\Z)\circ \phi_\xi)(i,p)$}\]
for each $(i,p) \in \hI_n$.
\end{Prop}

\begin{proof}
Recall that a height function $\xi \colon I_n \to \Z$ defines a Dynkin quiver $Q_\xi$ of type $\mathrm{A}_{n-1}$ so that we have an arrow $i \to j$ in $Q_\xi$ if and only if $|i-j| = 1$ and $\xi(i) > \xi(j)$.
For example, when $\xi$ is increasing, $Q_\xi$ is a monotone quiver: $Q_\xi = (1 \leftarrow 2 \leftarrow \cdots \leftarrow n-1)$.
Associated to a Dynkin quiver $Q$ of type $\mathrm{A}_{n-1}$, the rescaled dual PBW generator $\tE^*_Q(\alpha) \in U_n^+$ is defined for each positive root $\alpha \in R_n^+$ as in \cite[\S6.1]{HL}. 
(Note that $\tE^*_Q(\alpha)$ is written simply as $\tE^*(\alpha)$ in loc.\ cit.\ and our $t$ is equal to $v$ in loc.\ cit.)
For an arbitrary height function $\xi$, by the construction of the isomorphism $\Psi_\xi$ and \cite[Theorem 6.1]{HL}, we know that the isomorphism $\Psi_\xi \colon \hA_n \to K_t(\Cc_n)_{loc}$ sends $\iota_k(\tilde{E}^*_{Q_\xi}(\alpha))$ to $\chi_{q,t}(L(Y_{i,p}))$ with $(i,p) = \phi_\xi^{-1}(\alpha,k)$ for any $(\alpha,k) \in R_n^+ \times \Z$. 

Now, we assume that $\xi$ is increasing. 
Then, it is straightforward to check that the rescaled PBW generator $\tE_{Q_\xi}^*(\alpha)$ coincides with the element $e(\alpha)$ defined in \eqref{eq:ealpha} for any $\alpha \in R_n^+$.
Therefore, in this case, we have $\Psi_\xi(\iota_k(e(\alpha))) = \chi_{q,t}(L(Y_{i,p}))$ with $(i,p) = \phi_\xi^{-1}(\alpha,k)$ for any $(\alpha,k) \in R_n^+ \times \Z$. 
 
Thus, assuming that both $\xi$ and $\txi$ are increasing, we have
\begin{equation} \label{eq:main1}
\Psi_{\txi, \nu, \xi}(\chi_{q,t}(L(Y_{i,p}))) = \chi_{q,t}(L(\psi_{\txi,\nu,\xi}(Y_{i,p})))
\end{equation}
for any $(i,p) \in \hI_n$.  
In addition, since $(\alpha, \beta) = (\nu_*\alpha, \nu_*\beta)$ holds for any $\alpha, \beta \in R_n^+$, Lemma~\ref{Lem:FO} yields
\begin{equation} \label{eq:main2}
\Nn(Y_{i,p}, Y_{j,s}) = \Nn(\psi_{\txi,\nu,\xi}(Y_{i,p}), \psi_{\txi,\nu,\xi}(Y_{j,s}))
\end{equation}
for any $(i,p), (j,s) \in \hI_n$.
From \eqref{eq:main1} and \eqref{eq:main2}, together with Lemma~\ref{Lem:comm}, it follows that
\[ \Psi_{\txi, \nu, \xi}(E_t(m)) = E_t(\psi_{\txi, \nu, \xi}(m))\]
for all $m \in \cM_n^+$ (recall the definition of $E_t(m)$ in \eqref{eq:E}).
Moreover, it is immediate to see from the construction that $\Psi_{\txi, \nu, \xi}$ commutes with the bar-involutions. 
Therefore, for any $m \in \cM_n^+$, we find that the element $\Psi_{\txi, \nu, \xi}(\chi_{q,t}(L(m)))$ satisfies the two characterizing properties of $\chi_{q,t}(L(\psi_{\txi,\nu,\xi}(m)))$ in Theorem~\ref{Thm:canon}.
Thus, we obtain the desired equality \eqref{eq:main}.
\end{proof}

\begin{Rem}
For general height functions $\xi$ and $\txi$, the injective map $\psi_{\txi, \nu, \xi} \colon \cM_n^+ \to \cM_\tn^+$ is not a monoid homomorphism. 
It is only piecewise linear.  
See Remark~\ref{Rem:psi}.
\end{Rem}

By Theorem~\ref{Thm:main}, it makes sense to consider the specialization $\Psi_{\txi,\nu,\xi}|_{t=1}$ at $t^{1/2}=1$ of the homomorphism $\Psi_{\txi,\nu,\xi}$.
It is a unique ring homomorphism from $K(\Cc_n)$ to $K(\Cc_\tn)$ making the following diagram commute:
\begin{equation} \label{eq:t=1}
\vcenter{\xymatrix{
K_t(\Cc_n) \ar[r]^-{\Psi_{\txi,\nu,\xi}} \ar[d]_-{\ev_{t=1}} & K_t(\Cc_\tn) \ar[d]^-{\ev_{t=1}} \\ 
K(\Cc_n) \ar[r]^-{\Psi_{\txi, \nu, \xi}|_{t=1}} & K(\Cc_\tn).
}}
\end{equation} 

\begin{Cor} \label{Cor:t=1}
For any choice of $(\xi, \txi, \nu)$, the homomorphism $\Psi_{\txi, \nu, \xi}|_{t=1}$ induces an injective map from the set of simple classes of $\Cc_n$ to that of $\Cc_\tn$. 
More precisely, we have
\[ \Psi_{\txi, \nu, \xi}|_{t=1}([L(m)]) = [L(\psi_{\txi,\nu,\xi}(m))] \]
for any $m \in \cM_n^+$, where $\psi_{\txi,\nu,\xi} \colon \cM_n^+ \to \cM_\tn^+$ is as in Theorem~\ref{Thm:main}. 
\end{Cor}
\begin{proof}
The assertion follows from the commutativity of the diagram \eqref{eq:t=1} together with Theorem~\ref{Thm:Nak} and Theorem~\ref{Thm:main}.
\end{proof}

\subsection{A proof of Conjecture~\ref{Conj:BC}}
\label{Ssec:compare}

Let $n, \tn \in \Z_{>1}$ satisfying $\tn/n \in \Z$ as in \S\ref{Ssec:BC}.
We define the function $\nu_{\tn/n} \colon [1,n] \to [1,\tn]$ by $\nu_{\tn/n}(i) \seq \tn i /n$.
We also take the increasing height functions $\xi_0 \colon I_n \to \Z$ and $\txi_0 \colon I_\tn \to \Z$ given by $\xi_0(i) \seq i$ and $\txi_0(i) \seq i$ respectively.

Recall Brito--Chari's inflation $\Psi_{\tn, n} \colon K(\Cc_n) \to K(\Cc_\tn)$ and the monoid homomorphism $\psi_{\tn,n} \colon \cM_n^+ \to \cM_\tn^+$ from \S\ref{Ssec:BC}. 
The next proposition, together with Corollary \ref{Cor:t=1}, verifies Conjecture~\ref{Conj:BC}.
\begin{Prop}
We have $\Psi_{\txi_0, \nu_{\tn/n},\xi_0}|_{t=1} = \Psi_{\tn,n}$ and $\psi_{\txi_{0}, \nu_{\tn/n}, \xi_0} = \psi_{\tn,n}$.
\end{Prop}
\begin{proof}
A direct computation using the formula \eqref{eq:phi-1} in Example~\ref{Ex:phi} shows that the homomorphism $\psi_{\txi_0,\nu_{\tn/n}, \xi_0}$ sends $Y_{i,p}$ to $Y_{\tn i/n, \tn p/n}$ for all $(i,p) \in \hI_n$. 
Hence, we get the equality  $\psi_{\txi_0,\nu_{\tn/n}, \xi_0} = \psi_{\tn,n}$.
Then, by Corollary~\ref{Cor:t=1}, we have $\Psi_{\txi_0, \nu_{\tn/n}, \xi_0}|_{t=1}([L(Y_{i,p})]) = [L(Y_{\tn i/n, \tn p /n})]$ for any $(i,p) \in \hI_n$, which proves the equality $\Psi_{\txi_0, \nu_{\tn/n}, \xi_0}|_{t=1} = \Psi_{\tn,n}$. 
\end{proof}

\section{Categorification} \label{Sec:caton}
We end this paper with a brief discussion on a categorification of our quantum inflations using the quiver Hecke algebras of type $\mathrm{A}_\infty$.
It is built on the works of Kang--Kashiwara--Kim~\cite{KKK}, Kashiwara--Park~\cite{KP}, and Kashiwara--Kim--Oh--Park~\cite{KKOPbr}.

\subsection{Quiver Hecke algebras of type $\mathrm{A}_\infty$}
Let $P \seq \bigoplus_{a \in \Z}\Z \ep_{a}$ be a free abelian group of countably infinite rank equipped with the standard bilinear pairing $(\ep_a, \ep_b) = \delta_{a,b}$.
For each $a,b \in \Z$ with $a < b$, we set $\alpha_{a,b} \seq \ep_a - \ep_b$ and $\alpha_a \seq \alpha_{a,a+1}$.
The set $R^+ \seq \{ \alpha_{a,b} \mid a,b \in \Z, a<b\}$ is identical to the set of positive roots of type $\mathrm{A}_\infty$. 
Let $Q^+ \seq \sum_{a \in \Z} \Z_{\ge 0} \alpha_a \subset P$.

For each $\beta \in Q^+$, the quiver Hecke algebra $R(\beta)$ of type $\mathrm{A}_\infty$ is defined.
This is an associative $\Z$-graded algebra over $\C$ given by an explicit presentation. 
See \cite{KKK}.
Let $\Aa_\beta$ denote the category of finite-dimensional $\Z$-graded $R(\beta)$-modules and set $\Aa \seq \bigoplus_{\beta \in Q^+} \Aa_\beta$.
For $M \in \Aa_\beta$ and $M' \in \Aa_\beta'$, we can form their convolution product $M \circ M' \in \Aa_{\beta + \beta'}$. 
This gives rise to a bi-exact functor $\Aa_\beta \times \Aa_{\beta'} \to \Aa_{\beta + \beta'}$, which makes 
the category $\Aa$ into a $\Z$-graded monoidal abelian category.
In particular, the Grothendieck ring $K(\Aa) \seq \bigoplus_{\beta \in Q^+}K(\Aa_\beta)$
is a $\Z[t^{\pm 1}]$-algebra, where $t^{\pm 1}$ are the grading shifts. By convention, our $t$ is the same as $q^{-1}$ in \cite{KKK}. 
Taking the graded dual of the modules gives an anti-involution $M \mapsto M^*$ of $\Aa$.
It satisfies 
\begin{equation} \label{eq:dual}
(t^{\pm 1}M)^* \simeq t^{\mp 1}M^* \quad \text{and} \quad
(M \circ M')^* \simeq t^{-(\beta, \beta')}M'^* \circ M^*
\end{equation}
for $M \in \Aa_\beta$ and $M' \in \Aa_{\beta'}$.  
The Grothendieck ring $K(\Aa)$ has a free $\Z[t^{\pm 1}]$-basis formed by the classes of self-dual simple modules.    
Any simple module in $\Aa$ is isomorphic to a grading shift of a self-dual one.

For each positive root $\alpha \in R^+$, there is a self-dual one-dimensional module $L(\alpha) \in \Aa(\alpha)$.
In \cite{KKK}, $L(\alpha_{a,b})$ was denoted by $L(a,b-1)$.
The set $\{ [L(\alpha)] \mid \alpha \in R^+\}$ generates the $\Z[t^{\pm 1/2}]$-algebra $K(\Aa)$ and satisfies the relation
\begin{equation} \label{eq:La}
(1-t^{-2})[L(\alpha_{a,b})] = [L(\alpha_{a,b}) \circ L(\alpha_{b,c})] - t^{-1} [L(\alpha_{b,c}) \circ L(\alpha_{a,b})]
\end{equation}   
for any $a,b,c \in \Z$ with $a < b <c$.
See \cite[Proposition 4.3(vi)]{KKK}. 

\subsection{Categorification of $K_t(\Cc_n)$}
Given $n \in \Z_{>1}$, let $S_{n}$ de the automorphism of $P$ given by $S_n(\ep_a) = \ep_{a+n}$ for $a \in \Z$. 
Define a bilinear form $B_n(\cdot,\cdot) \colon P \times P \to \Z$ by
\[ B_n(x,y) \seq \sum_{k \in \Z_{>0}} (S_n^k(x),y). \] 
Then, twisting the convolution product $\circ$ in $\Aa$, we set 
\begin{equation} \label{eq:star}
M \star M' \seq t^{B_n(\beta, \beta')} M \circ M'
\end{equation}
for $M \in \Aa_\beta$ and $M' \in \Aa_{\beta'}$. 
This defines another monoidal structure on $\Aa$, and we write $\Aa_n \seq (\Aa, \star)$ for the $\Z$-graded monoidal abelian  category thus obtained.
From \eqref{eq:dual} and \eqref{eq:star}, we have
\begin{equation} \label{eq:star*}
(M \star M')^* \simeq t^{-(r_n(\beta), r_n(\beta'))} M'^* \star M^*,
\end{equation}
for $M \in \Aa_\beta$ and $M' \in \Aa_{\beta'}$, where $r_n \colon P \to P_n = \bigoplus_{i \in [1,n]}\Z \ep_i$ is the homomorphism given by $r_n(\ep_a) \seq \ep_{r_n(a)}$ in the notation of \eqref{eq:euclid}.

In \cite[\S4]{KKK}, a certain $\Z$-graded rigid monoidal abelian category $\Tt_n$ was constructed as a localization of the category $\Aa_n$.
By definition, it comes with an exact monoidal functor $\Omega_n \colon \Aa_n \to \Tt_n$ satisfying the following properties and it is universal among such monoidal functors: 
\begin{enumerate}
\item $\Omega_n(L(\alpha_{a,b})) = 0$ if $b-a > n$,
\item for any $a \in \Z$, there is an isomorphism between $\Omega_n(L(\alpha_{a,a+n}))$ and the unit object of $\Tt_n$ satisfying a certain compatibility condition (see \cite[Proposition A.12]{KKK} for detail).  
\end{enumerate}
Moreover, $\Omega_n$ sends simple modules in $\Aa_n$ to simple objects or zeros of $\Tt_n$, and every simple object of $\Tt_n$ is the image of a simple module in $\Aa_n$ under $\Omega_n$.
See \cite[Proposition A.11]{KKK}.  
We put $L_n(\alpha_{a,b}) \seq \Omega_n(L(\alpha_{a,b})) \neq 0$ for $a, b \in \Z$ with $b-a \in I_n = [1,n-1]$.
The monoidal category $\Tt_n$ gives a categorification of the quantum Grothendieck ring $K_t(\Cc_n)$ as follows.

\begin{Thm}[{Kang--Kashiwara--Kim~\cite{KKK}}] \label{Thm:SW}
There is an isomorphism 
\[ 
\Phi_n \colon K(\Tt_n)\otimes_{\Z[t^{\pm 1}]} \Z[t^{\pm 1/2}] \xrightarrow{\sim} K_t(\Cc_n)\]
of $\Z[t^{\pm 1/2}]$-algebras satisfying 
\begin{equation} \label{eq:Phi}
\Phi_n(t^{1/2}[L_n(\alpha_{a,b})]) = \chi_{q,t}(L(Y_{b-a, a+b}))\end{equation}
for any $a,b \in \Z$ with $b-a \in I_n$.
Moreover, there exists an exact monoidal functor  
$F_n \colon \Tt_n \to \Cc_n$
which sends simple objects of $\Tt_n$ to simple objects of $\Cc_n$ and satisfies $\ev_{t=1} \circ \Phi_n = [F_n]$.
\end{Thm}
\begin{proof}
This follows from \cite[Theorem 4.32]{KKK} and its proof.
The functor $F_n$ arises from the quantum affine Schur--Weyl duality. 
\end{proof}

Although the following corollary may be well-known for the experts, we shall give a detailed proof as the author could not find a proper reference. 

\begin{Cor}
The isomorphism $\Phi_n$ in Theorem~\ref{Thm:SW} sends each simple class in $K(\Tt_n)$ to a $t^{\frac{1}{2}\Z}$-multiple of a simple $(q,t)$-character in $K_t(\Cc_n)$.
\end{Cor}
\begin{proof}
As our base field $\C$ is of characteristic zero, the basis of the self-dual simple modules of $K(\Aa)$ correspond to the dual canonical basis of the quantum unipotent coordinate ring of type $\mathrm{A}_\infty$ by \cite{Rouq, VV11}.
In particular, the self-dual simple classes are parametrized by the set $(\Z_{\ge 0})^{\oplus R^+} $ (equivalent to the set of Zelevinsky's multi-segments) as follows. 
Define a total ordering $\le$ of the set $R^+$ so that we have $\alpha_{a,b} < \alpha_{a',b'}$ if $a < a'$, or $a = a'$ \& $b < b'$. 
For each $\bc = (c_{\alpha})_{\alpha \in R^+} \in (\Z_{\ge 0})^{\oplus R^+}$, we define 
\begin{equation} \label{eq:Mc}
M(\bc) = t^{-\sum_{\alpha \in R^+}c_{\alpha}(c_{\alpha}-1)/2} \mathop{\circ}_{\alpha \in R^+}^{\leftarrow} L(\alpha)^{\circ c_\alpha},
\end{equation}
where $\overset{\leftarrow}{\circ}$ means the ordered convolution product along the above ordering of $R^+$ increasing from the right to left. 
Then, the self-dual simple class $[L(\bc)]$ in $K(\Aa)$ corresponding to $\bc \in (\Z_{\ge 0})^{\oplus R^+}$ is characterized by the property 
\begin{equation} \label{eq:Lc} 
[L(\bc)] - [M(\bc)] \in \sum_{\wt(\bc') = \wt(\bc)} t^{-1} \Z[t^{-1}] [M(\bc')],
\end{equation}
where $\wt(\bc) \seq \sum_{\alpha \in R^+} c_\alpha \alpha \in Q^+$.
See \cite[Proposition 16]{LNT}.

Let $\sigma$ be the involution of $K(\Aa_n) \otimes_{\Z[t^{\pm 1}]} \Z[t^{\pm 1/2}]$ given by $\sigma(t^{1/2}) = t^{- 1/2}$ and $\sigma([M]) = t^{(r_n(\beta), r_n(\beta))/2}[M^*]$ for $M \in \Aa_\beta$. 
The equation \eqref{eq:star*} tells us that $\sigma$ is an anti-involution of $K(\Aa_n) \otimes_{\Z[t^{\pm 1}]} \Z[t^{\pm 1/2}]$. 
For $M \in \Aa_\beta$, we define its rescaled class $[\widetilde{M}] \seq t^{(r_n(\beta), r_n(\beta))/4}[M]$ so that $[\widetilde{M}]$ is fixed by $\sigma$ if and only if $[M] = [M^*]$.
We have 
\[[\widetilde{L(\alpha_{a,b})}] = t^{(1-\delta_{b-a,n})/2}[L(\alpha_{a,b})]\] 
for any $a, b \in \Z$ with $b-a \in [1,n]$.
Letting 
\[ \Phi_n' \seq \Phi_n \circ ([\Omega_n] \otimes 1) \colon K(\Aa_n) \otimes_{\Z[t^{\pm 1}]} \Z[t^{\pm 1/2}] \to K_t(\Cc_n),\] 
we have $\ol{(\cdot)} \circ \Phi'_n = \Phi'_n \circ \sigma$ by Theorem~\ref{Thm:SW}.
In particular, $\Phi'_n([\widetilde{L(\bc)}])$ is bar-invariant for any $\bc \in (\Z_{\ge 0})^{\oplus R^+}$.

Observe that, if $\alpha_{a,b} \ge \alpha_{a', b'}$ and $b-a, b'-a' \in [1,n]$, we have
\[B_n(\alpha_{a,b}, \alpha_{a',b'}) = -\delta(a = a' = b-n = b'-n),\]
and hence
\[ [\widetilde{M(\bc)}] = t^{\sum_{\alpha > \alpha'}c_{\alpha} c_{\alpha'}(r_n(\alpha), r_n(\alpha'))/2} \mathop{\star}^{\leftarrow}_{\alpha \in R^+}[\widetilde{L(\alpha)}]^{\star c_\alpha}\]
for any $\bc \in (\Z_{\ge 0})^{\oplus R^+}$ such that $c_{\alpha_{a,b}} = 0$ if $b-a > n$.
We also note that 
\[(r_n(\alpha_{a,b}), r_n(\alpha_{a',b'})) = -\Nn(Y_{b-a, a+b}, Y_{b'-a', a'+b'})\]
holds if $\alpha_{a,b} > \alpha_{a',b'}$ and $b-a, b'-a' \in I_n$ by Lemma~\ref{Lem:FO} and the computation in Example~\ref{Ex:phi}. 
Therefore, by Theorem~\ref{Thm:SW} and Lemma~\ref{Lem:comm}, we have
\[ \Phi'_n([\widetilde{M(\bc)}]) = \begin{cases}
0 & \text{if $c_{\alpha_{a,b}} \neq 0$ for some $a,b \in \Z$ with $b-a > n$}, \\
E_t(m_\bc) &\text{otherwise}
\end{cases}\]
for any $\bc \in (\Z_{\ge 0})^{\oplus R^+}$, where $m_\bc \seq \prod_{a,b \in \Z, b-a \in I_n}Y_{b-a, a+b}^{c_{\alpha_{a,b}}} \in \cM_n^+$ (recall the definition of $E_t(m)$ from \eqref{eq:E}).
Together with the equality \eqref{eq:Lc} and the bar-invariance of $\Phi'_n([\widetilde{L(\bc)}])$, we find that, if $\Phi'_n([\widetilde{L(\bc)}])$ is non-zero, it satisfies the two characterizing properties of $\chi_{q,t}(L(m_\bc))$ in Theorem~\ref{Thm:canon}, and hence $\Phi'_n([\widetilde{L(\bc)}]) = \chi_{q,t}(L(m_\bc))$. 
Since every simple object of $\Tt_n$ is the image under $\Omega_n$ of a simple module in $\Aa_n$, it verifies the assertion.
\end{proof}
\begin{Rem}
As far as the author understands, such a categorification of the quantum Grothendieck ring is known only for type $\mathrm{A}$ at this moment.  
One of the special features of type $\mathrm{A}$ is that the bozonic extension algebra $\hA_n$ can be obtained as an explicit quotient of the half of the quantized enveloping algebra of type $\mathrm{A}_{\infty}$, which enables us to construct the categorification out of the quiver Hecke algebras of type $\mathrm{A}_\infty$. 
It would be very interesting if we have a similar categorification for the other types.
\end{Rem}

\subsection{Categorification of quantum inflations}
Fix $n, \tn \in \Z_{>1}$ with $n < \tn$.
Let $\nu \colon [1,n] \to [1,\tn]$ be an increasing function as before. 
We extend $\nu$ to an increasing function $\nu \colon \Z \to \Z$ by 
\[ \nu(p) \seq k_n(p)\tn + \nu(r_n(p)) \]
for any $p \in \Z$, where we used the notation from \eqref{eq:euclid}.
Then, we have the induced homomorphism $\nu_* \colon P \to P$ given by $\nu_*(\ep_a) \seq \ep_{\nu(a)}$, which restricts to $R^+ \to R^+$ satisfying $\nu_*(\alpha_{a,b}) = \alpha_{\nu(a), \nu(b)}$ for any $a,b \in \Z$ with $a<b$.
Let us consider the collection $\{ L(\nu_*(\alpha_{a})) \mid a \in \Z\}$ of simple modules in $\Aa$, and apply the construction of Kashiwara--Park~\cite{KP} to the duality datum formed by the affinizations of these simple modules.
Note that the degree of the normalized $R$-matrix between $L(\nu_*(\alpha_{a}))$ and $L(\nu_*(\alpha_{b}))$ is equal to $-(\nu_*(\alpha_{a}), \nu_*(\alpha_{b})) = -(\alpha_a, \alpha_b)$ as computed in \cite[Proposition 4.3.(vii)]{KKK}.
The duality datum gives rise to a $\Z$-graded exact monoidal functor 
\[ \wh{F}_\nu \colon \Aa \to \Aa \]
satisfying $\wh{F}_\nu(L(\alpha_a)) \simeq L(\nu_*(\alpha_a))$ for any $a \in \Z$.
By \eqref{eq:La}, it follows that $\wh{F}_\nu(L(\alpha)) \simeq L(\nu_*(\alpha))$ for any $\alpha \in R^+$.
Since $B_{\tn}(\nu_*(x), \nu_*(y)) = B_n(x,y)$ for any $x,y \in P$, the functor $\wh{F}_\nu$ intertwines the twisted convolution products, and hence gives a $\Z$-graded exact monoidal functor 
\[\wh{F}_\nu \colon \Aa_n \to \Aa_\tn. \]
By the universality of the localizations (see the paragraph before Theorem \ref{Thm:SW}), there exists a unique $\Z$-graded exact monoidal functor
\[ F_\nu \colon \Tt_n \to \Tt_\tn\]
such that $F_\nu \circ \Omega_n \simeq \Omega_\tn \circ \wh{F}_\nu$. 
It satisfies $F_\nu(L_n(\alpha_{a,b})) \simeq L_\tn(\alpha_{\nu(a), \nu(b)})$ for any $a,b \in \Z$ with $b-a \in I_n$.

\begin{Prop} \label{Prop:Fnu}
Let $\xi_0 \colon I_n \to \Z$ and $\txi_0 \colon I_\tn \to \Z$ be the height functions given by $\xi_0(i) = i$ and $\txi_0(i)=i$ as in \S\ref{Ssec:compare}. 
For any increasing function $\nu \colon [1,n] \to [1,\tn]$, the following diagram commutes:
\begin{equation} \label{eq:Fnu}
\vcenter{
\xymatrix{
K(\Tt_n)\otimes_{\Z[t^{\pm 1}]} \Z[t^{\pm 1/2}] \ar[r]^-{\Phi_n}_-{\sim} \ar[d]_-{[F_\nu] \otimes 1}& K_t(\Cc_n) \ar[d]^-{\Psi_{\txi_0, \nu, \xi_0}} \\
K(\Tt_\tn)\otimes_{\Z[t^{\pm 1}]} \Z[t^{\pm 1/2}] \ar[r]^-{\Phi_\tn}_-{\sim}&  K_t(\Cc_\tn).
}}
\end{equation}
In particular, $F_\nu$ sends simple objects of $\Tt_n$ to simple objects of $\Tt_\tn$. 
\end{Prop}
\begin{proof}
In view of Theorem \ref{Thm:SW}, it is enough to prove that $\Psi_{\txi_0,\nu,\xi_0}$ sends $\chi_{q,t}(L(Y_{b-a, a+b}))$ to $\chi_{q,t}(L(Y_{\nu(b)-\nu(a), \nu(a)+\nu(b)}))$ for each $a,b \in \Z$ with $b-a \in I_n$.
This is easily checked by Proposition~\ref{Prop:main} and the explicit computation of $\phi_{\xi_0}$ and $\phi_{\txi_0}$ in Example~\ref{Ex:phi}.  
\end{proof}

\begin{Rem}
Specializing \eqref{eq:Fnu} at $t=1$, we get the commutative diagram:
 \[ 
\xymatrix{
K(\Tt_n)/(t-1)K(\Tt_n) \ar[r]^-{[F_n]}_-{\sim} \ar[d]_-{[F_\nu]}& K(\Cc_n) \ar[d]^-{\Psi_{\txi_0, \nu, \xi_0}|_{t=1}} \\
K(\Tt_\tn)/(t-1)K(\Tt_\tn) \ar[r]^-{[F_\tn]}_-{\sim}&  K(\Cc_\tn).
}
\]
The composition $F_\tn \circ F_{\nu} \colon \Tt_n \to \Cc_\tn$ can be constructed more directly following the recipe of \cite[\S6.1]{KKOPa} based on the generalized quantum affine Schur--Weyl duality in the sense of \cite{KKK}. In  \cite[Theorem 6.9]{KKOPa}, it is also proved  that the resulting functor always respects the simple classes. 
In this way, one may obtain an alternative proof of Conjecture~\ref{Conj:BC} without  using the quantum Grothendieck rings.  
\end{Rem}

Thus, the homomorphism $\Psi_{\txi_0, \nu, \xi_0}$ has a categorical lift $F_\nu$.
We also have a categorical lift of the homomorphism $\Psi_{\txi, \nu, \xi}$ for general height functions $\xi$ and $\txi$.
To see this, we recall the obvious factorization $\Psi_{\txi, \nu, \xi} = \Psi_{\txi, \txi_0} \circ \Psi_{\txi_0, \nu, \xi_0} \circ \Psi_{\xi_0, \xi}$.   
As we already mentioned in Remark \ref{Rem:FHOO2}, the automorphism $\Psi_{\xi_0, \xi}$ can be written as a composition of braid symmetries $\sigma_i$, $i \in I_n$, and a spectral parameter shift $\Sigma_c$ for some $c \in \Z$ mapping each $\chi_{q,t}(Y_{i,p})$ to $\chi_{q,t}(L(Y_{i,p+2c}))$.
By Kashiwara--Kim--Oh--Park~\cite[Theorem 3.3]{KKOPbr}, we know that the braid symmetry $\sigma_i$ has a categorical lift. 
As for the spectral parameter shift $\Sigma_c$, we have an obvious categorical lift, that is, the auto-equivalence of $\Tt_n$ transforming $L_n(\alpha_{a,b})$ to $L_n(\alpha_{a+c, b+c})$ for any $a,b \in \Z$ with $b-a \in I_n$.
Thus, by composition, the automorphism $\Psi_{\xi_0, \xi}$ has a categorical lift.
The same holds for $ \Psi_{\txi, \txi_0}$ as well.  
Thus, together with Proposition~\ref{Prop:Fnu}, we obtain the following conclusion. 
\begin{Thm} \label{Thm:categorification}
For any height functions $\xi \colon I_n \to \Z$, $\txi \colon I_\tn \to \Z$ and any increasing function $\nu \colon [1,n] \to [1,\tn]$, there exists a $\Z$-graded exact monoidal functor 
\[F_{\txi, \nu, \xi} \colon \Tt_n \to \Tt_\tn\] 
making the following diagram commute: 
\[ 
\xymatrix{
K(\Tt_n)\otimes_{\Z[t^{\pm 1}]} \Z[t^{\pm 1/2}] \ar[r]^-{\Phi_n}_-{\sim} \ar[d]_-{[F_{\txi, \nu, \txi}] \otimes 1}& K_t(\Cc_n) \ar[d]^-{\Psi_{\txi, \nu, \xi}} \\
K(\Tt_\tn)\otimes_{\Z[t^{\pm 1}]} \Z[t^{\pm 1/2}] \ar[r]^-{\Phi_\tn}_-{\sim}&  K_t(\Cc_\tn).
}
\]
In particular, $F_{\txi,\nu,\xi}$ sends simple objects of $\Tt_n$ to simple objects of $\Tt_\tn$. 
\end{Thm}

\subsection*{Acknowledgments}
The author would like to thank the anonymous referees for their useful comments and suggestions.
This work was supported by JSPS KAKENHI Grant No.\ JP23K12955.   

\subsection*{Declarations}
The author has no relevant financial or non-financial interests to disclose.
We do not analyze or generate any datasets, because our work proceeds within a mathematical approach.


\begin{thebibliography}{10}

\bibitem{BC}
Matheus Brito and Vyjayanthi Chari.
\newblock Higher order {K}irillov-{R}eshetikhin modules for {$\bold
  U_q(A_{n}^{(1)})$}, imaginary modules and monoidal categorification.
\newblock {\it J. Reine Angew. Math.}, 804:221--262, 2023.

\bibitem{CP}
Vyjayanthi Chari and Andrew Pressley.
\newblock {\it A guide to quantum groups}.
\newblock Cambridge University Press, Cambridge, 1994.

\bibitem{CP95}
Vyjayanthi Chari and Andrew Pressley.
\newblock Quantum affine algebras and their representations.
\newblock In {\it Representations of groups ({B}anff, {AB}, 1994)}, volume~16
  of {\it CMS Conf. Proc.}, pages 59--78. Amer. Math. Soc., Providence, RI,
  1995.

\bibitem{FM}
Edward Frenkel and Evgeny Mukhin.
\newblock Combinatorics of {$q$}-characters of finite-dimensional
  representations of quantum affine algebras.
\newblock {\it Comm. Math. Phys.}, 216(1):23--57, 2001.

\bibitem{FR}
Edward Frenkel and Nicolai Reshetikhin.
\newblock The {$q$}-characters of representations of quantum affine algebras
  and deformations of {$\mathscr W$}-algebras.
\newblock In {\it Recent developments in quantum affine algebras and related
  topics ({R}aleigh, {NC}, 1998)}, volume 248 of {\it Contemp. Math.}, pages
  163--205. Amer. Math. Soc., Providence, RI, 1999.

\bibitem{FHOO}
Ryo Fujita, David Hernandez, Se-jin Oh, and Hironori Oya.
\newblock Isomorphisms among quantum {G}rothendieck rings and propagation of
  positivity.
\newblock {\it J. Reine Angew. Math.}, 785:117--185, 2022.

\bibitem{FHOO2}
Ryo Fujita, David Hernandez, Se-jin Oh, and Hironori Oya.
\newblock Isomorphisms among quantum {G}rothendieck rings and cluster algebras.
\newblock Preprint, \arxiv{2304.02562}v2, 2023.

\bibitem{Her}
David Hernandez.
\newblock Algebraic approach to {$q,t$}-characters.
\newblock {\it Adv. Math.}, 187(1):1--52, 2004.

\bibitem{HL10}
David Hernandez and Bernard Leclerc.
\newblock Cluster algebras and quantum affine algebras.
\newblock {\it Duke Math. J.}, 154(2):265--341, 2010.

\bibitem{HLad}
David Hernandez and Bernard Leclerc.
\newblock Monoidal categorifications of cluster algebras of type {$A$} and
  {$D$}.
\newblock In {\it Symmetries, integrable systems and representations},
  volume~40 of {\it Springer Proc. Math. Stat.}, pages 175--193. Springer,
  Heidelberg, 2013.

\bibitem{HL}
David Hernandez and Bernard Leclerc.
\newblock Quantum {G}rothendieck rings and derived {H}all algebras.
\newblock {\it J. Reine Angew. Math.}, 701:77--126, 2015.

\bibitem{HO}
David Hernandez and Hironori Oya.
\newblock Quantum {G}rothendieck ring isomorphisms, cluster algebras and
  {K}azhdan-{L}usztig algorithm.
\newblock {\it Adv. Math.}, 347:192--272, 2019.

\bibitem{KKK}
Seok-Jin Kang, Masaki Kashiwara, and Myungho Kim.
\newblock Symmetric quiver {H}ecke algebras and {R}-matrices of quantum affine
  algebras.
\newblock {\it Invent. Math.}, 211(2):591--685, 2018.

\bibitem{KKOPbr}
Masaki Kashiwara, Myungho Kim, Se-jin Oh, and Euiyong Park.
\newblock Braid group action on the module category of quantum affine algebras.
\newblock {\it Proc. Japan Acad. Ser. A Math. Sci.}, 97(3):13--18, 2021.

\bibitem{KKOPa}
Masaki Kashiwara, Myungho Kim, Se-jin Oh, and Euiyong Park.
\newblock Cluster algebra structures on module categories over quantum affine
  algebras.
\newblock {\it Proc. Lond. Math. Soc. (3)}, 124(3):301--372, 2022.

\bibitem{KKOPbr2}
Masaki Kashiwara, Myungho Kim, Se-jin Oh, and Euiyong Park.
\newblock Braid symmetries on bosonic extensions.
\newblock Preprint, \arxiv{2408.07312}v1, 2024.

\bibitem{KP}
Masaki Kashiwara and Euiyong Park.
\newblock Affinizations and {R}-matrices for quiver {H}ecke algebras.
\newblock {\it J. Eur. Math. Soc. (JEMS)}, 20(5):1161--1193, 2018.

\bibitem{LNT}
Bernard Leclerc, Maxim Nazarov, and Jean-Yves Thibon.
\newblock Induced representations of affine {H}ecke algebras and canonical
  bases of quantum groups.
\newblock In {\it Studies in memory of {I}ssai {S}chur ({C}hevaleret/{R}ehovot,
  2000)}, volume 210 of {\it Progr. Math.}, pages 115--153. Birkh\"auser
  Boston, Boston, MA, 2003.

\bibitem{Li}
Yiqiang Li.
\newblock Quantum groups and edge contractions.
\newblock Preprint, \arxiv{2308.16306}v1, 2023.

\bibitem{Nak}
Hiraku Nakajima.
\newblock Quiver varieties and {$t$}-analogs of {$q$}-characters of quantum
  affine algebras.
\newblock {\it Ann. of Math. (2)}, 160(3):1057--1097, 2004.

\bibitem{Rouq}
Rapha\"el Rouquier.
\newblock Quiver {H}ecke algebras and 2-{L}ie algebras.
\newblock {\it Algebra Colloq.}, 19(2):359--410, 2012.

\bibitem{VV03}
M.~Varagnolo and E.~Vasserot.
\newblock Perverse sheaves and quantum {G}rothendieck rings.
\newblock In {\it Studies in memory of {I}ssai {S}chur ({C}hevaleret/{R}ehovot,
  2000)}, volume 210 of {\it Progr. Math.}, pages 345--365. Birkh\"auser
  Boston, Boston, MA, 2003.

\bibitem{VV11}
M.~Varagnolo and E.~Vasserot.
\newblock Canonical bases and {KLR}-algebras.
\newblock {\it J. Reine Angew. Math.}, 659:67--100, 2011.

\end{thebibliography}

\end{document}